\documentclass[12pt]{article}

\usepackage{amsmath,amsthm,amsfonts,amssymb,amscd,cite}
\usepackage{graphicx}

\allowdisplaybreaks[4]

\newtheorem {lemma}{Lemma}[section]
\newtheorem {theorem} {Theorem}[section]
\newtheorem {conjecture}{Conjecture}[section]

\usepackage[figurename=Fig.]{caption}
\allowdisplaybreaks [4]
\usepackage{fullpage}

\begin{document}

\title{Proof of a conjecture on distribution of Laplacian eigenvalues and diameter, and beyond}

\author{Leyou Xu\footnote{Email: leyouxu@m.scnu.edu.cn}, Bo Zhou\footnote{Corresponding author. Email: zhoubo@m.scnu.edu.cn}\\
School of Mathematical Sciences, South China Normal University\\
Guangzhou 510631, P.R. China}

\date{}
\maketitle

\begin{abstract}
Ahanjideh, Akbari,  Fakharan and  Trevisan proposed a conjecture  %in \cite{AhMS}
on the distribution of the Laplacian eigenvalues of graphs:  for any connected graph of order $n$ with diameter $d\ge 2$ that is not a path, the number of Laplacian eigenvalues in the interval $[n-d+2,n]$ is at most $n-d$.
We show that the conjecture is true, and give a complete characterization of graphs for which the conjectured bound is attained. This establishes
an interesting relation between the spectral and classical parameters. \\ \\
%and provide a family of graphs to show that the bound is best possible. This establishes
%an interesting relation between the spectral and classical parameters. \\ \\
\end{abstract}

\section{Introduction}

Let $G$ be a graph with vertex set $V(G)$ and edge set $E(G)$.
Let $G\cup H$ be the disjoint union of graphs $G$ and $H$. The disjoint union of $k$ copies of a graph $G$ is denoted by $kG$.
Denote by $P_n$ be the path of order $n$ and  $K_n$ the complete graph of order $n$.

For a graph $G$ with $S\subset V(G)$, denote by $G[S]$ the subgraph of $G$ induced by $S$ and $G-S$ the subgraph obtained from $G$ by deleting the vertices of $S$.

For a graph $G$ with $S\subseteq E(G)$, denote by $G-S$ the subgraph of $G$ obtained from $G$ by deleting all edges in $S$.
Particularly, if $S=\{e\}$, then we write it as $G-e$.
If $G'=G-S$ for some $S\subseteq E(G)$, then $G=G'+S$.

For $v\in V(G)$, denote by $N_G(v)$ the neighborhood of $v$ in $G$, and  $\delta_G(v)$ denotes the degree of $v$ in $G$.  For a graph $G$ of order $n$, the Laplacian matrix of  $G$  is the matrix
$L(G)=(\ell_{uv})_{u,v\in V(G)}$, where
\[
\ell_{uv}=\begin{cases}
\delta_G(u) & \mbox{if $u=v$,}\\
-1     & \mbox{if $u\ne v$ and $uv\in E(G)$,}\\
0     & \mbox{if $u\ne v$ and $uv\not\in E(G)$}.
\end{cases}
\]
That is, $L(G)$ is equal to the difference between the  diagonal
degree matrix and the adjacency matrix of $G$.
The  eigenvalues of $L(G)$ are known as the Laplacian eigenvalues of $G$, which we denote by $\mu_1(G), \dots, \mu_n(G)$, arranged in nonincreasing order. That is,
$\mu_j(G)$ is the $j$th (largest) Laplacian eigenvalue  of $G$ for $j=1,\dots, n$.
Note that any Laplacian eigenvalue of $G$ lies in the interval $[0,n]$ \cite{Me,Moh}.

%For a graph $G$ and  a Laplacian eigenvalue $\mu$ of  $G$,
%the multiplicity of $\mu$ is denoted by $m_G(\mu)$.
 For a graph $G$ of order $n$ and   an interval $I\subseteq [0,n]$, the number of Laplacian eigenvalues of $G$ in $I$ is denoted by $m_GI$.
Evidently, $m_G[0,n]=n$ for a graph $G$ of order $n$. But one is interesting in the knowledge how the Laplacian eigenvalues are distributed in $[0,n]$, which is not well understood.
The studying of the distribution of
Laplacian eigenvalues of graphs is a natural  problem that is relevant to
the many applications related to Laplacian matrices \cite{Me,Mo}.
There are a number of results that bound the number
of Laplacian eigenvalues in subintervals of $[0, n]$ (see, e.g., \cite{AhMS,BRT,CJT,GM,GMS,GWZF,GT,HJT,Me2,ZZD}).
Recently, Jacobs, Oliveira and Trevisan \cite{JOT} and Sin \cite{Sin} independently showed that $m_G[0,2-\frac{2}{n})\ge \lceil\frac{n}{2}\rceil$ (or equivalently, $m_G[2-\frac{2}{n},n]\le \lfloor\frac{n}{2}\rfloor$) if
$G$ is a tree of order $n$, which was conjectured in \cite{TCDV}. Here we note that $2-\frac{2}{n}$ is the average degree.
As pointed out by Jacobs, Oliveira and Trevisan in \cite{JOT},
it is also a hard problem because little is known about how the Laplacian eigenvalues are
distributed in the interval $[0, n]$.

For a connected graph $G$ with $u,v\in V(G)$,
the distance between $u$ and $v$ in $G$ is the length of a
shortest path connecting them in $G$.  The diameter of  $G$ is the greatest distance
between vertices in $G$. After showing that if $G$ is a connected graph of order $n$ with diameter $d\ge 4$, then $m_G(n-d+3,n]\le n-d-1$, Ahanjideh, Akbari,  Fakharan and  Trevisan \cite{AhMS} posed the following conjecture (see also Conjecture 1 and Problems 3 and 4 in \cite{AhMS}) and verified the validity for $d\in \{2,3\}$ or  when the diameter is less than or equal to the independence number.

\begin{conjecture} \cite{AhMS}\label{con}
Let  $G$ be a connected graph of order $n$ with diameter $d\ge 2$. If $G\ne P_{d+1}$,  then
\[
m_G[n-d+2,n]\le n-d.
\]
\end{conjecture}

In this paper, we show that the Conjecture \ref{con} is true.
 %and give a family of graphs to show the tightness of the conjectured bound.
Moreover, we give a complete characterization of graphs for which the conjectured bound is attained. The main result is as follows.

\begin{theorem}\label{x}
Let  $G$ be a connected graph of order $n$ with diameter $d\ge 2$. If $G\ne P_{d+1}$,  then
$m_G[n-d+2,n]\le n-d$ with equality if and only if\\
\begin{enumerate}
\item[(i)]
 $G\cong G_{n,d,t}$ for some $t$ with $2\le t\le \lfloor\frac{d}{2}\rfloor+1$, or

\item[(ii)]
 $G\cong G_{n,d,r,a}$ for some $r,a$ with $2\le r\le \lfloor\frac{d+1}{2}\rfloor$ and $1\le a\le n-d-2$.
\end{enumerate}
\end{theorem}

Here, the classes of graphs $G_{n,d,t}$ and $G_{n,d,r,a}$ are defined as follows:
For integers $n$, $d$ and $t$ with $2\le d\le n-2$ and $2\le t\le d$,  $G_{n,d,t}$ denotes the graph  obtained from a path $P_{d+1}:=v_1\dots v_{d+1}$ and a complete graph $K_{n-d-1}$ such that they are vertex disjoint by adding all edges connecting vertices of $K_{n-d-1}$ and vertices $v_{t-1}$, $v_t$ and $v_{t+1}$, see Fig. \ref{F1}.  For integers $n$, $d$, $r$ and $a$ with $3\le d\le n-2$, $2\le r\le d-1$ and $1\le a\le n-d-2$,
%$G_{n,d,r,a}$ denotes the graph obtained from a path $P_{d+1}:=v_1\dots v_{d+1}$ and a complete graph $K_{n-d-1}$ such that they are vertex disjoint by adding all edges connecting $a$ vertices in $K_{n-d-1}$ and vertices $v_{r-1}$, $v_r$ and $v_{r+1}$ and the remaining vertices in $K_{n-d-1}$ and  vertices $v_r$, $v_{r+1}$ and $v_{r+2}$, see Fig. \ref{F2}.
%
$G_{n,d,r,a}$ denotes the graph obtained from a path $P_{d+1}:=v_1\dots v_{d+1}$ and a complete graph $K_{n-d-1}$ such that they are vertex disjoint by adding all edges connecting
vertices of $K_{n-d-1}$ and vertices $v_r$ and $v_{r+1}$,
$a$ vertices in $K_{n-d-1}$ and vertex $v_{r-1}$, and the remaining vertices in $K_{n-d-1}$ and  vertex $v_{r+2}$, see Fig. \ref{F2}.

\begin{figure}[htbp]
\centering
\includegraphics [width =8cm]{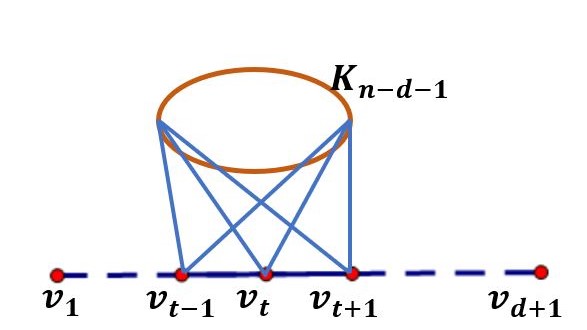}%\vspace{0.2cm}
\caption{The graph $G_{n,d,t}$.}
\label{F1}
\end{figure}

\begin{figure}[htbp]
\centering
\includegraphics [width =8cm]{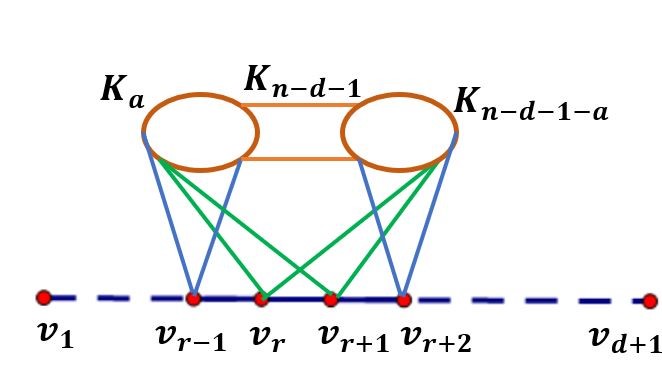}%\vspace{0.2cm}
\caption{The graph $G_{n,d,r,a}$.}
\label{F2}
\end{figure}

The rest of the paper is organized as follows. In section 2,  we introduce some basic notions, definitions, and important lemmas that will be  used. In section 3, we give a proof of Conjecture \ref{con}, and in section 4, we prove Theorem \ref{x}.

\section{Preliminaries}

The follow lemma
states about the $j$th Laplacian eigenvalue of a path, which appeared in \cite[p.~145]{AM} in plain text. In particular, $\mu_1(P_n)<4$.

\begin{lemma} \cite{AM} \label{FD}
$\mu_j(P_n)=4\sin^2\frac{(n-j)\pi}{2n}$ for $j=1,\dots, n$.
\end{lemma}

The following lemma follows from \cite[Corollary 2]{GM} and the subsequent comment, see \cite[p.~224]{GM}.

\begin{lemma}\label{d1}\cite{GM}
Let $G$ be a graph  on $n$ vertices with maximum degree $\Delta\ge 1$. Then $\mu_1(G)\ge \Delta+1$ with equality when $G$ is connected if and only if  $\Delta=n-1$.
\end{lemma}

%
%\begin{lemma} \cite[Theorem 3.6]{Moh} \label{C-e}
%If $\overline{G}$ denotes the complement of a graph $G$ of order $n\ge 2$, then $\mu_i(G)+\mu_{n-i}(\overline{G})=n$ for $i=1,\dots, n-1$.
%\end{lemma}
%
%By Lemma \ref{C-e}, we have $\sigma_L(K_n-e)=\{n^{[n-2]},n-2,0 \}$ for $n\ge 2$.

For an $n\times n$ Hermitian matrix $M$, $\rho_i(M)$ denotes its $i$-th largest eigenvalue of $M$ and $\sigma(M)=\{\rho_i(M): i=1,\dots,n \}$ is the spectrum of $M$.
For convenience, if $\rho$ is an eigenvalue of $M$ with multiplicity $s\ge 2$, then we write it as $\rho^{[s]}$ in $\sigma(M)$. For a graph $G$, let $\sigma_L(G)=\sigma(L(G))$.
Denote by $I_n$ the identity matrix of order $n$ and
$\mbox{diag}(m_1,\dots, m_n)$ the $n\times n$ diagonal matrix  with $(i,i)$-entry to be $m_i$ for $i=1,\dots, n$.

We need Weyl's inequalities \cite{We,KT}
with a characterization  of the equality cases
\cite[Theorem 1.3]{So}
(see also \cite[Theorem 4.3.1]{HJ}, where the eigenvalues are
considered nondecreasingly, while in this paper we consider them nonincreasingly).

\begin{lemma} \cite[Theorem 1.3]{So} \label{cw}
Let $A$ and $B$ be Hermitian matrices of order $n$.
For $1\le i,j\le n$ with $i+j-1\le n$,
\[
\rho_{i+j-1}(A+B)\le \rho_i(A)+\rho_j(B).
 \]
with equality if and only if there exists a nonzero vector $\mathbf{x}$ such that $\rho_{i+j-1}(A+B)=(A+B)\mathbf{x}$, $\rho_i(A)\mathbf{x}=A\mathbf{x}$ and $\rho_j(B)\mathbf{x}=B\mathbf{x}$.

%(ii) For $1\le i,j\le n$ with $i+j-n\ge 1$,
%\[
%\rho_i(A)+\rho_j(B)\le \rho_{i+j-n}(A+B)
% \]
%with equality if and only if there exists a nonzero vector $\mathbf{x}$ such that $\mathbf{x}$ is an eigenvector of $A$, $B$ and $A+B$ corresponding to $\rho_i(A)$, $\rho_j(B)$ and  $\rho_{i+j-n}(A+B)$ respectively.
\end{lemma}

We also need the following two types of interlacing theorem or inclusion principle.

\begin{lemma}\label{interlacing}\cite[Theorem 4.3.28]{HJ}
If $M$ is  a Hermitian matrix of order $n$ and $B$ is its principal submatrix of order $p$, then $\rho_{n-p+i}(M)\le\rho_i(B)\le \rho_{i}(M)$ for $i=1,\dots,p$.
\end{lemma}

\begin{lemma}\label{cauchy} \cite[Theorem 3.2]{Moh}
If $G$ is a graph with $e\in E(G)$, then
\[
\mu_1(G)\ge\mu_1(G-e)\ge \mu_2(G)\ge\dots\ge \mu_{n-1}(G-e)\ge \mu_n(G)=\mu_{n}(G-e)=0.
\]
\end{lemma}

If $\overline{G}$ denotes the complement of a graph $G$ of order $n\ge 2$, then \cite[Theorem 3.6]{Moh} $\mu_i(G)+\mu_{n-i}(\overline{G})=n$ for $i=1,\dots, n-1$.
So  $\sigma_L(K_n-e)=\{n^{[n-2]},n-2,0 \}$ for $n\ge 2$.

%Two matrices $A$ and $B$ are said to be permutationally similar if
%$B=P^\top A P$ for some permutation matrix $P$ (that is, $B$ is obtainable from $A$ by simultaneous
%permutations of its rows and columns). So two permutationally similar matrices have the same eigenvalues.

\begin{lemma}\label{gndt} For integers $n$, $d$ and $t$ with $2\le d\le n-2$ and $2\le t\le d$,
$m_{G_{n,d,t}}[n-d+2,n]= n-d$ and $\mu_{n-d}( G_{n,d,t})=n-d+2$.
\end{lemma}

\begin{proof}
Let $G=G_{n,d,t}$.
If $d=2$, then $G\cong K_n-e$ for some $e\in E(K_n)$, so the result follows as $\sigma_L(K_n-e)=\{n^{[n-2]},n-2,0\}$.

Suppose that $d\ge 3$. Note that $G$ is obtainable from vertex disjoint  $P_{d+1}:=v_1\dots v_{d+1}$ and  $K_{n-d-1}$  by adding all edges connecting vertices of $K_{n-d-1}$ and vertices $v_{t-1}$, $v_t$ and $v_{t+1}$. Let $v$ be an arbitrary vertex outside $P_{d+1}$.  Let $u_i=v_i$ for $i=1, \dots, t-1$, $u_t=v$ and $u_{i+1}=v_{i}$ for $i=t,\dots d+1$. Let $H=G[\{u_1,\dots,  u_{d+2}\}]$. Evidently,
$H$ is obtainable from the path $u_1\dots u_{d+2}$ by adding edges $u_{t-1}u_{t+1}$ and $u_tu_{t+2}$.
Then
\[
L(H)=L(P_{d+2})+S,
\]
where $S=(s_{ij})_{n\times n}$ with \[
s_{ij}=\begin{cases}
1&\mbox{ if }i=j\in\{t-1,t,t+1,t+2 \},\\
-1&\mbox{ if }\{i,j\}=\{t-1,t+1\},\{t,t+2\},\\
0&\mbox{ otherwise.}
\end{cases}
\]
It is easy to see that $\sigma(S)=\{2^{[2]},0^{[d]} \}$.
By Lemma \ref{cw},
\[
\mu_3(H)\le \mu_1(P_{d+2})+\rho_3(S)=\mu_1(P_{d+2})<4,
\]
Let $B$ be the principal submatrix of $L(G)$ corresponding to vertices of $H$.
Then
$B=L(H)+M$, where $M=\mbox{diag}(m_1,\dots,m_{d+2})$ with \[
m_i=\begin{cases}
n-d-2&\mbox{ if }i=t-1,t,t+1,t+2,\\
0&\mbox{ otherwise.}
\end{cases}
\]
By Lemmas \ref{interlacing} and \ref{cw},
\[
\mu_{n-d+1}(G)=\rho_{n-(d+2)+3}(L(G))\le \rho_3(B)\le \mu_3(H)+\rho_1(M)<4+n-d-2=n-d+2.
\]
So $m_G[n-d+2,n]\le n-d$.

On the other hand, the maximum degree of $G$ is $n-d+1<n-1$, we have by Lemma \ref{d1} that $\mu_1(G)>n-d+2$. As $(n-d+2)I_n-L(G)$ has $n-d$ equal rows, $n-d+2$ is a Laplacian eigenvalue of $G$ with multiplicity at least $n-d-1$. So
 $m_{G}[n-d+2,n]\ge n-d$.

 Now it follows that  $m_{G}[n-d+2,n]=n-d$. As $n-d+2$ is a Laplacian eigenvalue of $G$, we have
 $\mu_{n-d}(G)= n-d+2$.
\end{proof}

\begin{lemma}\label{gndla} For integers $n$, $d$, $t$ and $a$ with $3\le d\le n-2$, $2\le t\le d-1$ and $1\le a\le n-d-2$,
$m_{G_{n,d,t,a}}[n-d+2,n]=n-d$.
\end{lemma}

\begin{proof}
Let $G'=G_{n,d,t,a}$.
Note that $G'$ is obtainable  from vertex disjoint $P_{d+1}:=v_1\dots v_{d+1}$ and  $K_{n-d-1}$ by adding all edges connecting $a$ vertices in $K_{n-d-1}$ and vertices $v_{t-1}$, $v_t$ and $v_{t+1}$,  and the remaining vertices in $K_{n-d-1}$ and  vertices $v_t$, $v_{t+1}$ and $v_{t+2}$. Let $v$ be a neighbor of $v_{t-1}$ outside $P_{d+1}$. Let $u_i=v_i$ for $i=1, \dots, t-1$, $u_t=v$ and $u_{i+1}=v_{i}$ for $i=t,\dots d+1$. Let $H=G'[\{u_1,\dots,  u_{d+2}\}]$.
As in the proof of previous lemma, one gets $\mu_3(H)<4$.
Let $B$ be the principal submatrix of $L(G')$ corresponding to the vertices of $H$.
Then
$B=L(H)+M$, where $M=\mbox{diag}(m_1,\dots,m_{d+2})$ with \[
m_i=\begin{cases}
n-d-2&\mbox{ if }i=t,t+1,t+2,\\
a-1&\mbox{ if }i=t-1,\\
n-d-a-1&\mbox{ if }i=t+3,\\
0&\mbox{ otherwise.}
\end{cases}
\]
By Lemmas \ref{interlacing} and \ref{cw},
\[
\mu_{n-d+1}(G')\le\rho_3(B)\le \mu_3(H)+\rho_1(M)<4+n-d-2=n-d+2.
\]
So $m_{G'}[n-d+2,n]\le n-d$.

Next we show that $m_{G'}[n-d+2,n]\ge n-d$.
As the maximum degree of $G'$ is $n-d+1$,  less than $n-1$, we have by Lemma \ref{d1} that
 $\mu_1(G')>n-d+2$. As $(n-d+2)I_n-L(G')$ has $a+1$ and $n-d-a$ equal rows, respectively, $n-d+2$ is a Laplacian eigenvalue of $G'$ with multiplicity at least $n-d-1$. So $m_{G'}[n-d+2,n]\ge n-d$, as desired.
\end{proof}

%An $n\times n$ diagonal matrix $M$ with $(i,i)$-entry to be $m_i$ for $i=1,\dots, n$ is denoted by
%$\mbox{diag}(m_1,\dots, m_n)$.

%Given a graph $G$, for $v\in V(G)$ and a subgraph $F$ of $G$, let
%$N_F(v)=N_G(v)\cap V(F)$.
%A diametral path of a connected graph is a shortest path whose length is equal to the diameter of the graph.

A  diametral path of a connected graph with diameter $d$ is a path of $G$ that joins a pair of vertices with distance $d$. Given a graph $G$, for $v\in V(G)$ and a subgraph $F$ of $G$, let
$N_F(v)=N_G(v)\cap V(F)$.

\section{Proof of Conjecture \ref{con}}

We remark that that if $G$ is a path, then Conjecture \ref{con} is not true unless $n=3,4,5$. The reason is as follows: By Lemma \ref{FD}, %$\mu_j(P_n)=2-2\cos \frac{(n-j)\pi}{n}$ for $j=1,\dots, n$. So
\[
\mu_j(P_n)\ge 3 \Leftrightarrow 2-2\cos \frac{(n-j)\pi}{n}\ge 3 \Leftrightarrow
%\cos \frac{(n-j)\pi}{n}\le -\frac{1}{2}
%\Leftrightarrow
\frac{(n-j)\pi}{n}\ge \frac{2\pi}{3}\Leftrightarrow j\le \frac{n}{3}
\]
and hence if $n\ge 6$, then $m_{P_n}[3,n]\ge 2$. To prove Conjecture \ref{con} and overcome  difficulty  is that the assertion is not correct for paths,  we consider the number of vertices on a diametral path $P$ that are adjacent to each vertex outside $P$. The difficult case is when this number is at least two.
In this case, we consider the number of vertices on $P$ that are neighbors of some vertex outside $P$, which must be $2,3,4$ by the diameter condition. %So many subcases are considered.
%One main technique is to
%%decompose the edge set of the graph considered into several properly chosen sets and
%apply Weyl's inequality  and interlacing theorems to the corresponding matrices.

Now we give an affirmative answer to Conjecture \ref{con}. Though it follows from Theorem
\ref{x}, we give a simpler proof here.

\begin{theorem}\label{conx}
Let  $G$ be a connected graph of order $n$ with diameter $d\ge 2$. If $G\ne P_{d+1}$,  then
\[
m_G[n-d+2,n]\le n-d.
\]
\end{theorem}

\begin{proof} Note that
\[
m_G[n-d+2,n]\le n-d \Leftrightarrow \mu_{n-d+1}(G)<n-d+2.
\]
So, it suffices to show that $\mu_{n-d+1}(G)<n-d+2$.

%If $d=2,3$, the result follows \cite[p.~7]{AhMS}.  Suppose that $d\ge 4$.

Let $P:=v_1\dots v_{d+1}$ be a diametral path of $G$. As $G\ne P_{d+1}$, we have $d\le n-2$.
Let $B$ be the principal submatrix  of $L(G)$ corresponding to vertices on $P$. Then $B=L(P)+M$ with
\[
M=\mbox{diag}(\delta_G(v_1)-1, \delta_G(v_2)-2, \dots, \delta_G(v_d)-2, \delta_G(v_{d+1})-1).
\]
Note that $|V(G)\setminus V(P)|=n-d-1$.
If there is at most one vertex in $P$ adjacent to each vertex outside $P$ in $G$, then all entries $M_{ii}$ except one are at most  $n-d-2$, implying that $\rho_2(M)\le n-d-2$, so by Lemma \ref{interlacing} and Lemma \ref{cw}, we get
\begin{align*}
\mu_{n-d-1}(G)&=\rho_{n-(d+1)+2}(L(G))\\
&\le \rho_2(B)\\
&=\rho_2(L(P)+M)\\
&\le \rho_1(L(P))+\rho_2(M)\\
&= \mu_1(P_{d+1})+\rho_2(M)\\
%&=4\sin^2\frac{d\pi}{2(d+1)}+\rho_2(M)\\
&< 4+n-d-2\\
&=n-d+2,
\end{align*}
i.e., $\mu_{n-d+1}(G)<n-d+2$. %It thus follows that $m_G[n-d+2,n]\le n-d$.

%Let  $\Lambda=\max\{\delta_G(v_i): i=2,\dots, d\}$.
%Note that $N_G(v_i)\cap N_G(v_{j})=\emptyset$ for $\{i,j\}\subseteq \{1,\dots, d+1\}$ with $j-i\ge 3$,
%$\delta_G(v_1) \le n-(d+1)+1=n-d$, $\delta_G(v_{d+1}) \le n-d$, and
%$\Lambda \le n-(d+1)+2=n-d+1$.
%
%Suppose first that  $\Lambda\le n-d$.
%Assume that  $\delta_G(v_1)\ge \delta_G(v_{d+1})$. If $\delta_G(v_1)=\delta_G(v_{d+1})=n-d$, then $N_G(v_1)\cap N_G(v_{d+1})=V(G)\setminus V(P)$, so the diameter of $G$ is less than $d$, a contradiction.
%Thus $\delta_G(v_{d+1})\le n-d-1$.
%Note that $P$ is an induced subgraph of $G$.
%Let $B$ be the principal submatrix of $L(G)$ corresponding to the vertices of $P$.
%Then $B=L(P)+M$
%with
%\[
%M=\mbox{diag}(\delta_G(v_1)-1, \delta_G(v_2)-2, \dots, \delta_G(v_d)-2, \delta_G(v_{d+1})-1).
%\]
%Since $\delta_G(v_i)\le \Lambda\le n-d$ for $i=2,\dots,d$ and $\delta_G(v_{d+1})\le n-d-1$,
%we have $\rho_2(M)\le n-d-2$.
%By Lemma \ref{interlacing} and Lemma \ref{cw}, we get
%\begin{align*}
%\mu_{n-(d+1)+2}(G)&\le \rho_2(B)\\
%&\le \mu_1(P)+\rho_2(M)\\
%&=4\sin^2\frac{d\pi}{2(d+1)}+\rho_2(M)\\
%&\le 4\sin^2\frac{d\pi}{2(d+1)}+n-d-2\\
%&<n-d+2,
%\end{align*}
%i.e., $\mu_{n-d+1}(G)<n-d+2$. So $m_G[n-d+2,n]\le n-d$.

Next, suppose that there are at least two vertices in $P$ adjacent to each vertex outside $P$.
Let $v_t$ and $v_s$ be two such vertices with $1\le t<s\le d+1$ by requiring that $t$ is minimum.
As $P$ is a diametral path of $G$, any vertex outside $P$ is adjacent to at most three consecutive vertices on $P$.
So $s-t\le 2$.
Let $V_1=\cup_{u\in V(G)\setminus V(P)}N_P(u)$. If $s-t=1$ with $t\ge 2$ and $s\le d$, then $\{v_t, v_{t+1}\}\subseteq V_1\subseteq \{v_{t-1}, v_t, v_{t+1}, v_{t+2}\}$.
If $s-t=1$ with $t=1$ or $s=d+1$, then $\{v_1, v_{2}\}\subseteq  V_1\subseteq \{v_1, v_{2}, v_{3}\}$ or
$\{v_d, v_{d+1}\}\subseteq  V_1\subseteq \{v_{d-1}, v_{d}, v_{d+1}\}$.
If  $s-t=2$,  then $\{v_t, v_{t+2}\}\subseteq  V_1\subseteq \{v_t, v_{t+1}, v_{t+2}\}$.
So $|V_1|=2,3, 4$.

Suppose  that $|V_1|\le 3$.
Note that  $\{v_t, v_{t+1}\}\subseteq V_1\subseteq \{v_{t-1}, v_t, v_{t+1}\}$ with $t\ge 2$, or
$\{v_t, v_{t+1}\}\subseteq V_1\subseteq \{v_t, v_{t+1}, v_{t+2}\}$, or
$\{v_t, v_{t+2}\}\subseteq  V_1\subseteq \{v_t, v_{t+1}, v_{t+2}\}$.
Suppose without loss of generality that $V_1\subseteq  \{v_{\ell-1},v_\ell,v_{\ell+1}\}$ with $2\le \ell\le d$.
Then $G$ is a spanning subgraph of the graph $G_{n,d,\ell}$, which is obtained from $G$ by adding all possible edges between any two vertices
outside $P$ and between any vertex outside $P$ and any of $v_{\ell-1}, v_\ell$ and $v_{\ell+1}$.
Thus, by Lemmas \ref{cauchy} and \ref{gndt}, one gets
$\mu_{n-d+1}(G)\le \mu_{n-d+1}(G_{n,d,\ell})<n-d+2$.

Suppose next that $|V_1|=4$. Then $V_1=\{v_{t-1},v_t,v_{t+1},v_{t+2}\}$.
Recall that $v_t,v_{t+1}$ are adjacent to all vertices outside $P$.
Denote by $W_1$ the set of neighbors of $v_{t-1}$ outside $P$. Let $W_2= (V(G)\setminus V(P))\setminus W_1$.
Let $a=|W_1|$ %=\delta_G(v_{t-1})-\delta_P(v_{\ell-1})$
and $b=|W_2|$. Then $a+b=n-d-1$. By the definition of $V_1$ and choice of $v_t$, $a, b\ge 1$.
Note that $N_G(v_{t-1})\cap N_G(v_{t+2})=\emptyset$. Thus, $G-W_1$ is a spanning subgraph of $G_{n-a,d,t+1}$ with diameteral path $P$, and $G-W_2$ is a spanning subgraph of $G_{n-b,d,t}$ with diameteral path $P$. It then follows that
and $G$ is a spanning subgraph
%of the graph obtained from $G$ by adding all possible edges between any two vertices outside $P$ and between any vertex in $(V(G)\setminus V(P))\setminus W_1$ and
%$v_{t+2}$, which is
$G_{n,d,t,a}$ (with diameteral path $P$).
Now, by Lemmas \ref{cauchy} and \ref{gndla}, we have
$\mu_{n-d+1}(G)\le \mu_{n-d+1}(G_{n,d,t,a})<n-d+2$.

Now we have proved that $\mu_{n-d+1}(G)<n-d+2$ in any case,   so
$m_G[n-d+2,n]\le n-d$.
\end{proof}

\section{Proof of Theorem \ref{x}}

To prove Theorem \ref{x},  we consider the number of vertices on a diametral path $P$ that are adjacent to each vertex outside $P$. The difficult case is when this number is at least one.
In this case, we consider the maximum distance of vertices on $P$ that are neighbors of some vertex outside $P$, which is at most $4$. One main technique is to
apply Weyl's inequality  and interlacing theorems to the corresponding matrices.
 To simplify the proof of Theorem \ref{x}, we establish some lemmas.

\begin{lemma}\label{05}
Let $n, d$ and $t$ be integers with $3\le t\le d-2$ and $d\le n-3$.
For positive integers $a,b$ with $a+b=n-d-1$,
\begin{align*}
H_{n,d, t;a,b} &=P_{d+1}\cup K_{a}\cup K_b+\{wv_i:t-2\le i\le t,w\in V(K_a)\}\\
&\quad +\{wv_i:t\le i\le t+2,w\in V(K_b)\},
\end{align*}
where $P_{d+1}:=v_1\dots v_{d+1}$, see Fig. \ref{F3}.
 Then $\mu_{n-d}(H_{n,d, t;a,b})<n-d+2$.
\end{lemma}

\begin{figure}[htbp]
\centering
\includegraphics [width =8cm]{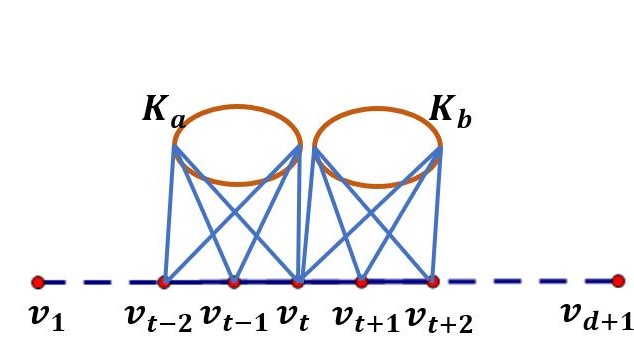}%\vspace{0.2cm}
\caption{The graph $H_{n,d, t;a,b}$.}
\label{F3}
\end{figure}

\begin{proof}
Let  $H=H_{n,d, t;a,b}-V(K_a)$. It is easy to see that $H\cong G_{n-a,d,t+1}$.
Let $B$ be the principal submatrix of $L(H_{n,d, t;a,b})$ corresponding to the vertices of $H$.
Then $B=L(H)+M$, where $M=\mbox{diag}(m_1,\dots,m_{n-a})$
with
\[
m_i=\begin{cases}
a&\mbox{ if }i=t-2,t-1,t,\\
0&\mbox{ otherwise.}
\end{cases}
\]
By Lemmas \ref{interlacing},  \ref{cw} and \ref{gndt}, one gets
\begin{align*}
\mu_{n-d}(H_{n,d, t;a,b}) &=\mu_{n-(n-a)+n-d-a}(H_{n,d, t;a,b})\\
&\le \rho_{n-d-a}(B)\\
& \le  \mu_{n-d-a}(G_{n-a,d,t+1})+\rho_{1}(M)\\
&=n-d-a+2+a=n-d+2.
\end{align*}
Suppose that $\mu_{n-d}(H_{n,d, t;a,b})=n-d+2$. Then
\[
\rho_{n-d-a}(B)
=\mu_{n-d-a}(G_{n-a,d,t+1})+a,
\]
so we have by Lemma \ref{cw} that
there exists a nonzero vector $\mathbf{x}$ such that
%\[
%B\mathbf{x}=(n-d+2)\mathbf{x},
% \]
 \[
 L(H)\mathbf{x}=(n-d-a+2)\mathbf{x}
 \]
and
\[
M\mathbf{x}=a\mathbf{x}.
\]
From $M\mathbf{x}=a\mathbf{x}$,   we have $x_w=0$ for each $w\in V(H)\setminus\{v_{t-2},v_{t-1},v_{t}\}$. From $L(H)\mathbf{x}=(n-d-a+2)\mathbf{x}$ at $v_{t+1}$, we have \[
(n-d-a+1)x_{v_{t+1}}-x_{v_t}-x_{v_{t+2}}-\sum_{w\in V(K_b)}x_w=(n-d-a+2)x_{v_{t+1}},
\]
so $x_{v_t}=0$. Similarly, from $L(H)\mathbf{x}=(n-d-a+2)\mathbf{x}$ at $v_t$ and $v_{t-1}$, respectively, we have $x_{v_{t-1}}=x_{v_{t-2}}=0$. Thus, $\mathbf{x}$ is a zero vector, which is a contradiction. It follows that $\mu_{n-d}(H_{n,d, t;a,b})<n-d+2$.
\end{proof}

\begin{lemma}\label{5}
Let $n,d$ and $t$ be integers with $3\le t\le d-2$ and $d\le n-3$.
For positive integers $a,b$ and $c$  with  $a+b+c=n-d-1$,
let
\begin{align*}
H_{n,d, t;a,b,c} &=P_{d+1}\cup ((K_{a}\cup K_b)\vee K_c)+\{wv_i: t-2\le i\le t,w\in V(K_a)\}\\
&\quad +\{wv_i:t\le i\le t+2,w\in V(K_b)\}\\
&\quad +\{wv_i:t-1\le i\le t+1,w\in V(K_c)\},
\end{align*}
where $P_{d+1}:=v_1\dots v_{d+1}$, see Fig. \ref{F4}.  Then $\mu_{n-d}(H_{n,d, t;a,b,c})<n-d+2$.
\end{lemma}

\begin{figure}[htbp]
\centering
\includegraphics [width =8cm]{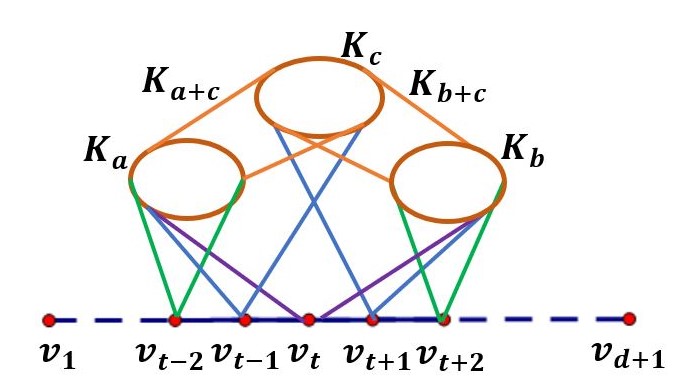}%\vspace{0.2cm}
\caption{The graph $H_{n,d, t;a,b,c}$.}
\label{F4}
\end{figure}

\begin{proof}
Let $H=H_{n,d, t;a,b,c}-V(K_c)$. Then
$H\cong H_{n-c,d,t;a,b}$.
Let $B$ be the principal submatrix of $L(H_{n,d, t;a,b,c} )$ corresponding to the vertices of $H$.
Note that $B=L(H)+\mbox{diag}(s_1,\dots,s_{n-c})$,
where
\[
s_i=\begin{cases}
c&\mbox{ if }i=t-1,t,t+1,\\
0&\mbox{ otherwise.}
\end{cases}
\]
It hence follows from Lemmas \ref{interlacing}, \ref{cw} and \ref{05} that
\begin{align*}
\mu_{n-d}(H_{n,d, t;a,b,c}) &\le \rho_{n-c-d}(B)\\
&\le \mu_{n-c-d}(H_{n-c,d,t;a,b})+c\\
&<n-c-d+2+c=n-d+2,
\end{align*}
as desired.
\end{proof}

%\begin{lemma}\label{hnt1}
%Let $n$ and $t$ be integers with $1\le t\le n-2$.
%Let $P_n:=v_1\dots v_n$ and
%\[
%H_{n,t}^{(1)}=P_n+v_tv_{t+2}.
%\]
%Then $\mu_{2}(H_{n,t}^{(1)})<4$.
%\end{lemma}
%\begin{proof}
%By Lemma \ref{cw},
%$\mu_{2}(H_{n,t}^{(1)})\le \mu_1(P_n)<4$.
%\end{proof}

For integers  $n$ and $t$ with $1\le t\le n-3$, we denote by $P_{n,t}^{++}$ the graph obtained from $P_{n-1}=v_1\dots v_{n-1}$ by  adding a new vertex $u$ and two new edges $uv_t$ and $uv_{t+2}$.

\begin{lemma}\label{hnt2}
Let $n$ and $t$ be integers with $1\le t\le n-3$. Then $\mu_2(P_{n,t}^{++})<4$.
\end{lemma}
\begin{proof} %Let $H=P_{n-1}+\{uv_i:i=t,t+2\}$.

If $t=1$, then $P_{n,t}^{++}-uv_3\cong P_n$, so $\mu_2(P_{n,t}^{++})\le \mu_1(P_n)<4$  by Lemma \ref{cauchy}.
Similarly, if $t=n-3$, then $P_{n,t}^{++}-uv_t\cong P_n$ and so $\mu_2(P_{n,t}^{++})<4$.
Suppose in the following that $2\le t\le n-4$.

Let $H=P_{n,t}^{++}+uv_{t+1}$. Then $H\cong G_{n,n-2,t+1}$, so we have by Lemma \ref{gndt} that $\mu_2(H)=4$.
Let $u=v_n$. Then $L(H)=L(P_{n,t}^{++})+S$,
where $S=(s_{ij})_{n\times n}$ with \[
s_{ij}=\begin{cases}
1&\mbox{ if }i=j=t+1,n,\\
-1&\mbox{ if }\{i,j\}=\{t+1,n\},\\
0&\mbox{ otherwise.}
\end{cases}
\]
It is evident that $\sigma(S)=\{2,0^{[n-1]}\}$. It hence follow by Lemma \ref{cw}  that
$\mu_2(P_{n,t}^{++})\le \mu_2(H)+\rho_1(-S)= \mu_2(H)=4$.

Suppose that $\mu_2(P_{n,t}^{++})=4$. Then there exists a nonzero vector $\mathbf{x}$ such that
$S\mathbf{x}=0$ and
$L(P_{n,t}^{++})\mathbf{x}=4\mathbf{x}=L(H)\mathbf{x}$. Let $x_i=x_{v_i}$ for $i=1,\dots,n$. From $S\mathbf{x}=0$, one gets $x_{t+1}=x_n$.

%Next, we show that $x_i=(-1)^{i-1}(2i-1)x_1$ for $i=2,\dots,t$, by induction on $i$.

From $L(P_{n,t}^{++})\mathbf{x}=4\mathbf{x}$ at $v_1$, we have
$x_1-x_2=4x_1$, so $x_2=-3x_1$.
Suppose that
$x_j=(-1)^{j-1}(2j-1)x_1$ for each $j\le i$ with $i=2,\dots,t-1$. From $L(P_{n,t}^{++})\mathbf{x}=4\mathbf{x}$ at $v_i$, we have $2x_i-x_{i-1}-x_{i+1}=4x_i$, so
$x_{i+1}=-2x_i-x_{i-1}=(-1)^{i-1}(2i-1)x_1$. This shows that  $x_i=(-1)^{i-1}(2i-1)x_1$ for $i=2,\dots,t$.
From $L(H)\mathbf{x}=4\mathbf{x}$ at $v_t$, we have
$3x_t-x_{t-1}-2x_{t+1}=4x_t$,
so $x_{t+1}=(-1)^{t}\cdot x_1$.
From $L(P_{n,t}^{++})\mathbf{x}=4\mathbf{x}$ at $v_{t+1}$, we have
$2x_{t+1}-x_{t}-x_{t+2}=4x_{t+1}$,
so
$x_{t+2}=(-1)^{t}(2t-3)x_1$.

Similarly,
$x_i=(-1)^{n-i-1}(2(n-i)-1)x_{n-1}$ for $i=t+2,\dots,n-2$ and $x_{t+1}=(-1)^{n-t}\cdot x_{n-1}$.
It follows that
$x_{n-1}=(-1)^nx_1$. Thus  $x_{t+2}=(-1)^{n-t-3}(2(n-t-2)-1)x_{n-1}=(-1)^{t+1}(2(n-t-2)-1)x_1$.
Therefore, $(-1)^{t+1}(2(n-t-2)-1)x_1=(-1)^{t}(2t-3)x_1$, i.e., $(n-4)x_1=0$, i.e.,
$x_1=0$.  So $\mathbf{x}=0$, a contradiction. Thus $\mu_2(P_{n,t}^{++})<4$.
\end{proof}

%\begin{lemma}\label{q1}
%Let $t$ and $n$ be integers with $2\le t\le n-3$. Let $P_{n-1}:=v_1\dots v_{n-1}$ and \[
%Q_{n,t}^{(1)}=P_{n-1}+v_tv_{t+2}+\{uv_i:t-1\le i\le t+2\},
%\]
%where $u$ is the vertex outside $P_{n-1}$. Then $\mu_3(H_{n,d,t})<4$.
%\end{lemma}
%\begin{proof}
%Note that \[
%Q_{n,t}^{(1)}=H_{n-1,t}^{(1)}+\{uv_i:t-1\le i\le t+2\}.
%\]
%Let $B$ be the principle submatrix of $L(Q_{n,t}^{(1)})$ by deleting row and column corresponding to $u$. Then it follows from Lemmas \ref{interlacing}, \ref{cw} and \ref{hnt1} that \[
%\mu_3(Q_{n,t}^{(1)})\le \rho_2(B)\le \mu_2(H_{n-1,t})+1<4+1=5,
%\]
%as desired.
%\end{proof}

%\begin{lemma}\label{q2}
%Let $t$ and $n$ be integers with $1\le t\le n-4$. Let $P_{n-2}:=v_1\dots v_{n-2}$, $P_2=uv$ and \[
%Q_{n,t}^{(2)}=P_{n-2}\cup P_2+\{uv_i:i=t,t+2\}+\{vv_i:i=t,t+1,t+2\},
%\]
%Then $\mu_3(Q_{n,t}^{(2)})<5$.
%\end{lemma}
%\begin{proof}
%It's easy to see that \[
%Q_{n,t}^{(2)}=H_{n-1,t}^{(2)}+\{vv_i:i=t,t+1,t+2\}+uv.
%\]
%Let $B$ be the principle submatrix of $L(Q_{n,t}^{(2)})$ by deleting row and column corresponding to $v$. Then by Lemmas \ref{interlacing}, \ref{cw} and \ref{hnt2}, we have\[
%\mu_3(Q_{n,t}^{(2)})\le \rho_2(B)\le \mu_2(P_{n,t}^{++})+1<4+1=5,
%\]
%as desired.
%\end{proof}

\begin{lemma}\label{hq2}
Let $n,d$, $t$  and $a$ be integers with $2\le d\le n-3$ and $2\le t\le d$.
Let $G=G_{n,d,t}$ with diametral path $P=v_1\dots v_{d+1}$.
Let $w$ be a vertex outside $P$.  Then $\mu_{n-d}(G-v_iw)<n-d+2$ for $i=t-1,t,t+1$.
\end{lemma}
\begin{proof}
Let $H$ be the subgraph of $G-v_{t-1}w$ induced by $V(P)\cup \{w\}$, which is obtainable from the path $P_{d+2}=v_1\dots v_twv_{t+1}\dots v_{d+1}$ by adding an edge $v_tv_{t+1}$. It follows from Lemma \ref{cauchy} that $\mu_2(H)\le \mu_1(P_{d+2})<4$. Let $B$  be the principal submatrix of $L(G-v_{t-1}w)$ corresponding to vertices of $H$. Then by Lemmas  \ref{interlacing} and \ref{cw},
\begin{align*}
\mu_{n-d}(G-v_{t-1}w)&=\mu_{n-(d+2)+2}(G-v_{t-1}w)\\
&\le \rho_{2}(B)\\
&\le \mu_2(H)+n-d-2\\
&<4+n-d-2=n-d+2.
\end{align*}
The proof of $\mu_{n-d}(G-v_{t+1}w)<n-d+2$ is similar.

Let $H'$ be the subgraph of $G-v_tw$ induced by $V(P)\cup \{w\}$, which is obtainable from the path $P_{d+1}=v_1\dots v_{d+1}$ by adding edges $v_{t-1}w$ and $v_{t+1}w$. By Lemma \ref{hnt2}, $\mu_2(H)<4$. Let $B'$ be the principal submatrix of $L(G-v_tw)$ corresponding to vertices of $H'$. Then by Lemmas \ref{interlacing} and \ref{cw},
\begin{align*}
\mu_{n-d}(G-v_{t}w)&=\mu_{n-(d+2)+2}(G-v_{t}w)\\
&\le \rho_{2}(B')\\
&\le \mu_2(H')+n-d-2\\
&<4+n-d-2=n-d+2. \qedhere
\end{align*}
\end{proof}

\begin{lemma}\label{hq1}
Let $n,d$, $t$  and $a$ be integers with
 $3\le d\le n-3$, $2\le t\le d-1$ and $1\le a\le n-d-2$.
Let $G=G_{n,d,t,a}$ with diametral path $P=v_1\dots v_{d+1}$.
If $e$ is an edge joining one of $v_{t-1}, v_t, v_{t+1}, v_{t+2}$ and one vertex outside $P$, then
$\mu_{n-d}(G-e)<n-d+2$.
\end{lemma}

\begin{proof} %There are four cases: $e=v_{t-1}w, v_{t+2}w, v_tw,v_{t+1}w$.
Suppose first that $e=v_{t-1}w$. Let $H$ be the subgraph of $G-e$ induced by $V(P)\cup\{w\}$, which is obtainable
from the path  $P_{d+2}=v_1\dots v_twv_{t+1}\dots v_{d+1}$ by adding edge $v_tv_{t+1}$. By Lemma \ref{cauchy}, $\mu_{2}(H)\le \mu_1(P_{d+2})<4$.
Let $B$  be the principal submatrix of $L(G-v_{t-1}w)$ corresponding to  vertices of $H$.
By Lemmas \ref{interlacing} and \ref{cw},
\begin{align*}
\mu_{n-d}(G-e)&=\mu_{n-(d+2)+2}(G-e)\\
&\le \rho_{2}(B)\\
&\le \mu_2(H)+n-d-2\\
&<4+n-d-2=n-d+2.
\end{align*}
If $e=v_{t+2}w$, then similar argument leads to  $\mu_{n-d}(G-e)<n-d+2$.

Next, suppose that $e=v_tw$ with $w\in N_G(v_{t-1})$. Let $H'$ be the subgraph of $G-e$ induced by $V(P)\cup\{w\}$.  It is evident that
$H'\cong P_{d+2,t-1}^{++}$. By Lemma \ref{hnt2}, $\mu_{2}(H')<4$.
Let $B'$ be the principal submatrix of $L(G-e)$ corresponding to  vertices of $H'$.
By Lemmas \ref{interlacing} and \ref{cw},
\begin{align*}
\mu_{n-d}(G-e)&=\mu_{n-(d+2)+2}(G-e)\\
&\le \rho_{2}(B')\\
&\le \mu_2(H)+n-d-2\\
&<4+n-d-2=n-d+2.
\end{align*}
Similar arguments applies to the remaining cases: $e=v_tw$ with $w\in N_G(v_{t+2})$, $e=v_{t+1}w$ with $w\in w\in N_G(v_{t-1})$, and
$e=v_{t+1}w$ with $w\in w\in N_G(v_{t+1})$.
\end{proof}

Now,  we are ready to give a proof of  Theorem \ref{x}.

%\begin{theorem}\label{x}
%Let  $G$ be a connected graph of order $n$ with diameter $d\ge 2$. If $G\ne P_{d+1}$,  then
%$m_G[n-d+2,n]=n-d$ if and only if\\
%\begin{enumerate}
%\item[(i)]
% $G\cong G_{n,d,t}$ for some $t$ with $2\le t\le d$, or
%
%\item[(ii)]
% $G\cong G_{n,d,\ell,a}$ for some $\ell,a$ with $2\le \ell\le d-1$ and $1\le a\le n-d-2$.
%\end{enumerate}
%\end{theorem}

\begin{proof}[Proof of Theorem \ref{x}]
By Lemmas \ref{gndt} and \ref{gndla}, if $G$ satisfies (i) or (ii) in Theorem \ref{x}, then $m_G[n-d+2,n]=n-d$.

Suppose that $G$ does not satisfy (i) and (ii) in Theorem \ref{x}.	
That is,
 $G\ncong G_{n,d,t}$ for any $t$ with $2\le t\le d$, and
 $G\ncong G_{n,d,r,a}$ for any $r,a$ with $2\le r\le d-1$ and $1\le a\le n-d-2$.
It suffices to show that
 $m_G[n-d+2,n]<n-d$. As
\[
m_G[n-d+2,n]< n-d \Leftrightarrow \mu_{n-d}(G)<n-d+2,
\]
it suffices to show that $\mu_{n-d}(G)<n-d+2$.

Let $P:=v_1\dots v_{d+1}$ be a diametral path of $G$. As $G\ne P_{d+1}$, we have $d\le n-2$.
Note that $L(P)+M$ is the principal submatrix of $L(G)$ corresponding to vertices on $P$, where
\[
M=\mbox{diag}(\delta_G(v_1)-1, \delta_G(v_2)-2, \dots, \delta_G(v_d)-2, \delta_G(v_{d+1})-1).
\]
If there is no vertex in $P$ adjacent to each vertex outside $P$ in $G$, then each entry $M_{ii}$ is at most
$n-d-2$,  so
$\rho_1(M)\le n-d-2$, and   by Lemmas \ref{interlacing} and  \ref{cw} , one gets
\begin{align*}
\mu_{n-(d+1)+1}(G)%&=\rho_{n-(d+1)+1}(L(G))\\
&\le \rho_1(L(P)+M)\\
&\le  \mu_1(P_{d+1})+\rho_1(M)\\
%&=4\sin^2\frac{d\pi}{2(d+1)}+\rho_1(M)\\
&< 4+n-d-2\\
&=n-d+2,
\end{align*}
i.e., $\mu_{n-d}(G)<n-d+2$.
Assume that there is at least one vertex in $P$ adjacent to each vertex outside $P$.
Let $v_t$ be  such a vertex.
As $P$ is a diametral path of $G$, any vertex outside $P$ is adjacent to at most three consecutive vertices on $P$.
Let $V_1=\cup_{u\in V(G)\setminus V(P)}N_P(u)$, $\alpha=\min\{s:v_s\in V_1\}$ and $\beta=\max\{s:v_s\in V_1\}$.
Then $t-\alpha\le 2$ and $\beta-t\le 2$. So $\beta-\alpha\le 4$.

\noindent
{\bf Case 1.}  $\beta-\alpha=4$.

In this case, $\alpha=t-2$ and $\beta=t+2$. Let $N_1=N_G(v_{t-2})\setminus V(P)$,  $N_2=N_G(v_{t+2})\setminus V(P)$ and $N_3=V(G)\setminus (V(P)\cup N_1\cup N_2)$. Set $a=|N_1|$, $b=|N_2|$ and $c=|N_3|=n-d-1-a-b$. As the diameter of $G$ is $d$, we have
$N_1\cap N_2=\emptyset$, $N_G(v_{t-1})\cap N_2=\emptyset$
$N_G(v_{t+1})\cap N_1=\emptyset$, and  any vertex from $N_1$ is not adjacent to any vertex from $N_2$.
Thus,  $G[N_1\cup N_3]$ is a spanning subgraph of $K_{a+c}$,
$G[N_2\cup N_3]$ is a spanning subgraph of $K_{b+c}$,
$N_G(v_{t-1})\setminus V(P)\subseteq N_1\cup N_3$, and $N_G(v_{t+1})\setminus V(P)\subseteq N_2\cup N_3$.
If $c=0$, then $G$ is a spanning subgraph
of a graph obtained from $P\cup (K_a\cup K_b)$ with a diametral path $P$ by adding all possible edges to join each of  $v_{t-2}, v_{t-1}, v_{t}$ on $P$ and all vertices of $K_a$ outside $P$,
and edges to join each of  $v_{t}, v_{t+1}, v_{t+2}$ on $P$ and all vertices of $K_b$ outside $P$. This graph is $H_{n,d, t;a,b}$ given in Lemma \ref{05}.
It then follows by Lemmas \ref{cauchy} and \ref{05} that $\mu_{n-d}(G)\le \mu_{n-d}(H_{n,d, t;a,b})<n-d+2$, as desired.
Otherwise, $c\ge 1$, and $G$  is a spanning subgraph
of a graph obtained from $P\cup ((K_a\cup K_b)\vee K_c)$ with a diametral path $P$ by adding all possible edges to join  $v_{t-2}$ and  all vertices of $K_a$ outside $P$, $v_{t+2}$ and  all vertices of $K_b$ outside $P$, $v_{t-1}$ and  all vertices of $K_a\cup K_c$ outside $P$, $v_{t+1}$ and  all vertices of $K_b\cup K_c$ outside $P$, and $v_{t}$ and  all vertices outside $P$. This graph is $H_{n,d, t;a,b,c}$ given in Lemma \ref{5}.
It then follows by Lemmas  \ref{cauchy} and \ref{5} that $\mu_{n-d}(G)\le \mu_{n-d}(H_{n,d, t;a,b,c})<n-d+2$, as desired.

\noindent
{\bf Case 2.} $\beta-\alpha=3$.

Obviously, $\{t-\alpha, \beta-t\}=\{1,2\}$.
 Assume that $V_1\subseteq \{v_{t-1},v_t,v_{t+1},v_{t+2}\}$.
Let $W_1=N_G(v_{t-1})\setminus V(P)$.
Let $a=|W_1|$ %=\delta_G(v_{\ell-1})-\delta_P(v_{\ell-1})$
and $b=|V(G)\setminus (V(P)\cup W_1)|=n-d-1-a$.
By the definition of $\alpha$ and $\beta$, we have $a, b\ge 1$.
As the diameter of $G$ is $d$, $N_G(v_{t+2})\cap W_1=\emptyset$, so
$N_G(v_{t+2})\setminus V(P)\subseteq V(G)\setminus (V(P)\cup W_1)$.
Thus $G$ is a spanning subgraph of the graph obtained from $P\cup K_{n-d-1}$ with a diametral path $P$ by adding all possible edges to join  $v_{t-1}$  and vertices of $W_1$,
$v_{t+2}$  and vertices of $V(G)\setminus (V(P)\cup W_1)$, and any of $v_t, v_{t+1}$ and vertices outside $P$, which is $G_{n,d,t,a}$. By assumption, $G\ncong G_{n,d,t,a}$. So $G$ is a spanning subgraph of $G_{n,d,t,a}-e$ for some $e\in E(G_{n,d,t,a})$. As $G$ is connected, $G_{n,d,t,a}-e$ is connected.
As $G$ is a spanning subgraph of $G_{n,d,t,a}-e$, we have by Lemma \ref{cauchy} that
$\mu_{n-d}(G)\le \mu_{n-d}(G_{n,d,t,a}-e)$. So, it suffices to show that
$\mu_{n-d}(G_{n,d,t,a}-e)<n-d+2$.
For convenience, let $G'=G_{n,d,t,a}$.

If $e$ joins two vertices in $W_1$, then $G'-e\cong G'-v_tw$ for some $w\in W_1$.
If $e$ joins two vertices in $V(G')\setminus (V(P)\cup W_1)$, then $G'-e\cong G'-v_{t+1}w$ for some $w\in V(G')\setminus (V(P)\cup W_1)$.
If $e$ joins a vertex in $W_1$ and a vertex in $V(G')\setminus (V(P)\cup W_1)$, then $G'-e\cong G'-v_{t+1}w$ for some $w\in W_1$.
If $e$ lies on $P$, then it is only possible for $e=v_{j}v_{j+1}$ with $j=t-1,t, t+1$. Note that $G'-v_{t-1}v_t\cong G'-v_{t-1}w$, $G'-v_{t}v_{t+1}\cong G'-v_{t+1}w$ and $G'-v_{t+1}v_{t+2}\cong G'-v_{t+2}z$
for some $w\in W_1$ and $z\in V(G')\setminus (V(P)\cup W_1)$.
So, we may assume that $e$ joins one of $v_{t-1},\dots,v_{t+2}$ and a vertex outside $P$.
Therefore, the result follows from Lemma \ref{hq1}.

%If  $e=v_{t-1}w$ for some $w\in W_1$ or $e=v_{t+2}w$ for some $w\in V(G)\setminus (V(P)\cup W_1)$, then we have by Lemma \ref{hq1} (i) that $\mu_{n-d}(G'-e)<n-d+2$. We are left with the case that  $e=v_tw, v_{t+1}w$ for some $w\in V(G)\setminus V(P)$. In either case,  $G'-e\cong Q^{n,d,t}$, so we have by Lemma \ref{hq1} (ii) that $\mu_{n-d}(G'-e)<n-d+2$.

\noindent
{\bf Case 3.}  $\beta-\alpha\le 2$.

Note that $V_1\subseteq \{v_{t-1},v_t,v_{t+1}\}$.
Then $G$ is a  spanning subgraph of the graph with diametral path $P$ obtained by adding all possible edges  with at most one end vertex $v_{t-1},v_t,v_{t+1}$ on $P$. This graph is definitely
$G_{n,d,t}$. By assumption, $G\ncong G_{n,d,t}$. So $G$ is a  spanning subgraph of $G_{n,d,t}-e$ for some $e\in E(G_{n,d,t})$. Note that $G_{n,d,t}-e$ is connected.
It suffices to show that $\mu_{n-d}(G'-e)<n-d+2$, where
 $G'=G_{n,d,t}$.

If $e$ joins two vertices outside $P$, then $G'-e\cong G'-v_tw$, where $w\in V(G')\setminus V(P)$.
If $e$ lies on $P$, then $e=v_{t-1}v_t, v_tv_{t+1}$, and $G'-e\cong G'-v_{t-1}w, G'-v_{t+1}w$, where $w\in V(G')\setminus V(P)$.
So we may assume $e=v_{t-1}w, v_tw, v_{t+1}w$ for some vertex $w$ outside $P$.

If $d\le n-3$, then the result follows from Lemma \ref{hq2}.
Suppose that $d=n-2$.
If $e=v_{t-1}w$, then $P_n=v_1\dots v_twv_{t+1}\dots v_{n-1}$ is a path in $G'-e$ and
$G'-e-E(P_n)$ has a single edge $v_tv_{t+1}$. By Lemma \ref{cauchy},  $\mu_2(G'-e)\le \mu_1(P_n)<4$.
Similar argument applies if $e=v_{t+1}w$.
If $e=v_tw$, then $G'-e\cong P_{n,t-1}^{++}$,  so we have $\mu_{n-d}(G'-e)<n-d+2$ by Lemma \ref{hnt2}.

%\noindent
%{\bf Case 3.2.} $d\le n-3$.
%
%
%If $e=v_{t-1}w$, then $G'-e\cong H_{n,d,t}$, so we have by Lemma \ref{hq1} that
%$\mu_{n-d}(G'-e)<4+n-d-2=n-d+2$.
%Similar argument applies if $e=v_{t+1}w$.
%If $e=v_tw$, then $G'-e\cong H^{n,d,t}$ and so by Lemma \ref{hq2}, $\mu_{n-d}(G'-e)<n-d+2$.

Combining the above three cases,  we have $\mu_{n-d}(G)<n-d+2$, so
$m_G[n-d+2,n]< n-d$. This completes the proof.
\end{proof}

\bigskip

\noindent {\bf Acknowledgement.}
%The authors  would like to thank the reviewers for constructive comments and suggestions.
This work was supported by the National Natural Science Foundation of China (No.~12071158).

\end{document}